\documentclass{amsart}

\usepackage{graphicx} % Required for inserting images
\usepackage{amsmath}
\usepackage{amssymb}
\usepackage{amsthm}
\usepackage{graphicx}
\usepackage{comment}
\usepackage[colorlinks=true, allcolors=blue]{hyperref}

\newtheorem{thm}{Theorem}[section]

\newtheorem{lem}[thm]{Lemma}

\newtheorem{claim}[thm]{Claim}
\newtheorem{rmk}[thm]{Remark}

\numberwithin{equation}{section}

\title {A note on Turán numbers and the Erdos-Stone-Simonovits theorem}
\author{Stefan Gobej}
\address{Vlaicu Voda High School\\
Curtea de Arges, Romania}
\email{stefangobej@gmail.com}
\date{\today}

\begin{document}

\begin{abstract}
Given a fixed graph $H$, we say that a graph $G$ is \textit{$H$-free} if $G$ does not contain $H$ 
as a subgraph. The Tur\'an number $ex(n, H)$ of $H$ is the maximum number of edges in 
an $n$-vertex $H$-free graph.  The study of Tur\'an number of graphs is a central topic in extremal graph theory. The purpose of this article is to present some well-known results about this field but also to prove the Erdős-Stone-Simonovits theorem in an original manner.

\end{abstract}

\maketitle

\tableofcontents

\section{Introduction}

In this short note, we survey a few classical results from extremal graph theory, with a particular focus on Turán-type problems. Our aim is to revisit foundational theorems and offer a fresh perspective on their proofs. Central to this note is a new proof we present for the Erdős–Stone–Simonovits theorem.

Traditionally, this theorem is proved using Szemerédi’s Regularity Lemma \cite{szemeredi1975}. Our approach avoids this, relying instead on more elementary combinatorial tools. We believe this alternative proof offers greater transparency and may be more accessible to readers unfamiliar with deeper machinery.

In addition to Erdős–Stone–Simonovits, we include discussions and proofs of Turán’s theorem and the Kővári–Sós–Turán theorem as well as the Bondy-Simonovits theorem.

\emph{Notation and terminology}
All graphs considered in this
paper are simple. We denote a simple graph by $G=(V(G),E(G))$ where $V(G)$ is the set of vertices and $E(G)$ is the set of edges. Let $\delta(G)$ and $\Delta(G)$ denote the minimum degree and maximum degree of $G$. For any subset $S$ of $V(G)$ we denote by $G\backslash S$ the subgraph induced by $V(G)\backslash S$. 
The chromatic number $\chi(G)$ of a graph $G$ is the smallest natural number $c$ such that the vertices of $G$ can be colored with $c$ colors and no two vertices of the same color are adjacent.

Throughout the paper, we use the standard Big O notation to describe asymptotic upper bounds. For a function $f(x)$
, we write $f(x) = O(g(x))$ if there exists some $C > 0$ such that $|f(x)|\leqslant  Cg(x) $ for all $x$.

%\begin{defn}

\section{Avoiding cliques}

 The basic statement of extremal graph theory is Mantel’s theorem, proved in 1907\cite{mantel1907}, which states that any graph on $n$ vertices without a triangle (i.e. $K_3-$free)  has at most $\frac{n^2}{4}$ edges.

%\subsection{Mantel's theorem}

\begin{thm}[Mantel\cite{mantel1907}]\label{thm-Mantel}

If a graph $G$ with $n$ vertices contains no triangle, then it has at most $\frac{n^2}{4}$ edges:
\[ex(n,K_3)\leqslant \frac{n^2}{4}. \]
\end{thm}

\begin{proof}[First proof of Theorem~\ref{thm-Mantel}]
Suppose that $G$ has $m$ edges. Let $v$ and $u$ be two vertices in $G$ that are joined by an edge. 

If $deg(x)$ is the degree of a vertex $x$, we see that $deg(v)+ deg(u)\leqslant n$. This is because every vertex in the graph $G\backslash \{ v,u\}$ is connected to at most one of $v$ and $u$. 

Note now that 
\[ \sum_{v\in V(G)} (deg(v))^2=\sum_{uv\in E(G)} (deg(u)+deg(v))\leqslant n m .\]

By Cauchy-Schwarz inequality we have 
\[\sum_{v\in V(G)} deg(v)^2\geqslant \frac{1}{n} \left( \sum_{v\in V(G)} deg(v) \right)^2 =\frac{4m^2}{n}\]
so we will have that $nm\geqslant \frac{4m^2}{n}$ and the result follows.    
\end{proof}

\begin{proof}[Second proof of Theorem~\ref{thm-Mantel}]
Let $A\subset V(G)$ be a maximum independent set of $G$.
 Consider the complement $B=V(G) \backslash A$ of the subset $A$. Every edge must have a vertex in $B$ because $A$ is a maximal independent set of $G$. Thus
\[|E(G)|\leqslant \sum_{v\in B} deg(v)\leqslant |A|\cdot|B|\leqslant \left(\frac{|A|+|B|}{2}\right)^2=\frac{n^2}{4}.\]
\end{proof}
 
 The natural generalization of this theorem to cliques of size $r$ is the following, proved by Paul Tur\'an in 1941\cite{turan1941}.

\begin{thm}[Tur\'an\cite{turan1941}]\label{thm- Turán}

If a graph $G$ on $n$ vertices contains no copy of $K_{r+1}$, the complete graph on $r+1$ vertices, then it contains at most
$\frac{n^2}{2}(1-\frac{1}{r})$ edges:
\[ex(n,K_{r+1})\leqslant \frac{n^2}{2}\left(1-\frac{1}{r}\right).\]

\end{thm}

\begin{proof}[Proof of Theorem~\ref{thm- Turán}]

We argue by induction on $n$. The theorem is trivially true for $n=1,2,\ldots ,r$. We will therefore assume that it is true for all values less than $n$ and prove it for $n$. Let $G$ be a graph on $n$ vertices that contains no $K_{r+1}$ and has the maximum possible number of edges. Then $G$ contains copies of
$K_r$. Otherwise, we could add edges to $G$, contradicting the maximality of the number of edges of $G$.

Let $H$ be a clique of size $r$ of the graph $G$ and let $T$ be its complement. Since $T$ has $n-r$ vertices and does not contain $K_{r+1}$, there are at most 
\[ \frac{(n-r)^2}{2}(1-\frac{1}{r})\]
edges in $T$. Moreover, since every vertex in $T$ can have at most $r-1$ neighbors in $H$, the number of edges between $H$ and $T$ is at most $(n-r)(r-1)$. Summing, we see
that in $G$ we have at most 
$$\frac{r^2-r}{2}+(n-r)(r-1)+\frac{(n-r)^2}{2}(1-\frac{1}{r})=\frac{n^2}{2}(1-\frac{1}{r})$$
edges.
\end{proof}

\section{Avoiding bipartite graphs}
We are now going to begin an in-depth study of the extremal number for bipartite graphs. The following result was proved by K\'ov\'ari-S\'os-Tur\'an in 1954 \cite{kovari1954}

\begin{thm}[K\'ov\'ari-S\'os-Tur\'an\cite{kovari1954}]\label{thm-Kovari-Sos-Turan}

For any natural numbers $r$ and $t$ with $r\leqslant t$, there exists a constant $c$ such that:
$$ex(n,K_{r,t}) \leqslant cn^{2-\frac{1}{r}}$$

\end{thm}

\begin{proof}[Proof of Theorem~\ref{thm-Kovari-Sos-Turan}]

Let $G$ be a $K_{r,t}$-free graph with $n$ vertices and $m$ edges.

 Assume that all vertices in $G$ have degree at least $r-1$. We will count the number of $K_{1,r}$.

 Firsty, for any set $A$ of $r$ vertices, count the number of $K_{1,r}$ which use $A$ as the part of size $r$. There is one for each common neighbor of $A$. But $A$ has less than $t$ common neighbors, otherwise $A$ and $t$ of their common neighbors form a $K_{r,t}$. So there are at most $t-1$ of these
 $K_{1,r}$ implying that the number of $K_{1,r}$ in G is at most 
 \[(t-1)\cdot \binom{n}{r}.\]
 And we have that 
 \[(t-1)\cdot \binom{n}{r}\leqslant \frac{n^r(t-1)}{r!}.\]
 So the number of $K_{1,r}$ in G is smaller than   
 \[\frac{n^r(t-1)}{r!}.\]

 Now, for each vertex $v$, count the number of $K_{1,r}$ which use $v$ as the part of size $1$. There is one for each set of $s$ neighbors of $v$ for a total of 
\[\binom{\deg(v)}{r}.\]

 Thus the number of $K_{1,r}$ in $G$ is 
 \[\sum_{v\in V(G)}\limits \binom{deg(v)}{r}\geqslant n \binom{\frac{2m}{n}}{r}\geqslant \frac{(\frac{2m}{n}-r+1)^rn}{ r!}\]
 by convexity and the Handshake Lemma. So we have that 
 \[\frac{(\frac{2m}{n}-r+1)^rn}{ r!}<\frac{n^r(t-1)}{r!}.\]

 Thus 
 \[m\leqslant \frac{1}{2}(t-1)^{\frac{1}{r}}n^{2-\frac{1}{r}}+\frac{n}{2}(r-1)=O(n^{2-\frac{1}{r}}).\]

 Now, assume that $G$ has vertices of degree less than $r-1$. Consider the graph $G_1$ formed by adding arbitrary edges to each vertex $v$ with $deg(v) <s-1$ until $deg(v)=s-1$. The new graph $G_1$ is $K_{r,t}$-free so it satisfies $|E(G_1)|\leqslant O(n^{2-\frac{1}{r}})$ from the previous work. Since $|E(G)|\leqslant |E(G_1)|$.

\end{proof}

%   \subsection{Remark}
\begin{rmk}
   Because every bipartite graph is a subgraph of a complete bipartite graph, K\a'ov\a'ari-S\a'os-Tur\a'an gives an upper bound on $ex(n,H)$ for every bipartite graph $H$.
\end{rmk}

\section{Avoiding cycles}
We will consider the extremal problem for some of the most obvious examples of bipartite graphs, cycles of even length. The main theorem we will prove is the upper bound $ex(n, C_{2k})\leqslant cn^{1+\frac{1}{k}}$, due to Bondy and Simonovits \cite{bondy340095}.

\begin{thm}[Erdős\cite{erdos1964}]\label{thm-Erdős}
\[ ex(n,C_4)\leqslant \frac{n}{4}(\sqrt{4n-3}+1)\]
\end{thm}

\smallskip

\begin{proof}[Proof of the Theorem~\ref{thm-Erdős}]
Let $G$ be a $C_4$-free graph with $n$ vertices and $m$ edges.
Let $P$ be the number of paths of length $2$ in $G$.

Firstly, each vertex $v$ is the middle vertex of $deg(v)\choose 2$ paths of length $2$ in $G$.
\[P=\sum_{v\in V(G)}  \binom{deg(v)}{2} =\frac{1}{2}\sum_{v\in V(G)}\limits (deg(v))^2-\frac{1}{2}\sum_{v\in V(G)}\limits deg(v)\geqslant \frac{2 m^2}{n}-m\]
by Cauchy-Schwarz inequality and Handsake Lemma.

Secondly, each unordered pair of vertices has the endpoints of at most one path of length $2$. So we have that $P\leq \binom{n}{2}$.

Thus 
\[\binom{n}{2}\geqslant \frac{2 m^2}{n}-m \]
so we have that 
\[n^3-n^2\geqslant 4m^2-2mn.\]
Consequently we have that  
\[ \frac{n^2}{4}(4n-3)\geqslant (2m-\frac{n}{2})^2. \]
In conclusion 
\[m\leqslant \frac{n}{4}(\sqrt{4n-3}+1).\]
\end{proof}

\begin{thm}[Bondy-Simonovits\cite{bondy340095}]\label{thm-Bondy-Simonovits}
For all natural numbers $k \geqslant 2$, there exists a constant $c = c(k)$, such that 
\[ex(n,C_{2k})\leqslant cn^{1+\frac{1}{k}},\]
for sufficiently large $n$ (depending on $k$), where $C_{2k}$ means a simple cycle on $2k$ vertices.
\end{thm}

\begin{proof}[Proof of the Theorem]~\ref{thm-Bondy-Simonovits}

\begin{claim}\label{claim:1}
Every graph $G$ has a subgraph whose minimum degree is at least half the average degree of $G$.
\end{claim}

Suppose that $H$ is a $C_{2k}$ free graph on $n$ vertices ($n$ sufficiently large) with $|E(H)|>cn^{1+\frac{1}{k}}$. Then
\[d(H)= \frac{2|E(H)|}{n} > 2cn^{\frac{1}{k}},\]
where $d(H)$ denotes the average degree of $H$. 
By Claim~\ref{claim:1}, $H$ must
contain a subgraph $G$ of minimum degree 
\[\delta(G) \geqslant \frac{d(H)}{2}>cn^{\frac{1}{k}}.\]
Pick a vertex $v\in V(G)$ and perform breadth-first search starting from $v$. Define $V_0 = \{ v\}$ and $V_i = \{ u \in V(G)| l(v,u) = i\}$ for $1\leqslant  i\leqslant  k$, where $l(x, y)$ denotes the distance between $x$ and
$y$ in the tree.

Denote by $G[V_i]$ the induced subgraph on the vertex set $V_i$ and by $G[V_i, V_{i+1}]$ the bipartite subgraph induced by partitions $V_i$ and $V_{i+1}$

The main proof relies on the fact that both $G[V_i]$ and $G[V_i, V_{i+1}]$ are sparse, for all  $i$. This is formalized in the following lemma.

%
%Observation 1.1
 \begin{claim}\label{claim:1.1}
 There exist constants $c_1$ in terms of $k$ and $c_2$ in terms of $k$ such that for $1\leqslant  i \leqslant  k-1$ the following
hold:
\begin{enumerate}
    \item $d(G[V_i])\leqslant  c_1k$;  (here $d(G[V_i])$ denotes the average degree of $G[V_i]$)
    \item $d(G[V_i,V_{i+1}])\leqslant  c_2k$;   (here $d(G[V_i,V_{i+1}])$ denotes the average degree of $G[V_i,V_{i+1}]$)
\end{enumerate}
\end{claim}

\begin{claim}\label{claim:2}
For $1\leqslant  k \leqslant  n$ we will have 
\[ \frac{|V_{i+1}|}{|V_i|}\geqslant  \frac{c}{2c_2}n^{\frac{1}{k}}.\]
\end{claim}

\begin{proof}[Proof of \ref{claim:2}]
We denote $|V_i|$ by $n_i$ for every $1\leqslant  k \leqslant  n$.
We induct on $i$. In the base case $i=0$, 
\[\frac{n_1}{n_0} =\frac{deg(v)}{n}\geqslant  \delta(G) > c_0n^{\frac{1}{k}}.\]
Thus if $c_2 \geqslant  \frac{1}{2}$, we are through.
For $1\leqslant  i\leqslant  k -1$, we have that:
\begin{align*}
|E(G[V_i,V_{i+1}])| & =\sum_{u\in V_i}\limits (deg(u)-deg_i(u)-deg_{i-1,i}(u)) = \\
& =(\sum_{u\in V_i}\limits deg(u))-2|E(G[V_i])|-|E(G[V_{i-1},V_i])| \geqslant \\
& \geqslant n_i\delta(G)-n_id(G[V_i])-\frac{1}{2}d(G[V_{i-1},V_i])(n_{i-1}+n_i)\geqslant \\
& \geqslant n_icn^{\frac{1}{k}}-n_ic_1k-\frac{1}{2}(n_{i-1}+n_i)c_2k \geqslant \\
& \geqslant n_icn^{\frac{1}{k}}-n_ic_1k-n_ic_2k\geqslant  \frac{c}{2}n_in^{\frac{1}{k}}
\end{align*}
for n sufficiently large, where the second inequality uses Claim~\ref{claim:1.1} and fact 
that $n_i\geqslant  n_{i-1}$ which comes by induction and that $n$ is large. Here $deg_i(u)$ denotes the degree of $u$ to vertices in $V_i$ and $deg_{i-1,i}(u)$ denotes its degree
to vertices in $V_{i-1}$.
Also we have  that:

\[|E(G[V_i,V_{i+1}])|=\frac{1}{2}d(G[V_i,V_{i+1}])(n_i+n_{i+1})\leqslant  \frac{1}{2}(n_i+n_{i+1})c_2k\] 
by Claim~\ref{claim:1.1}.

Combining both inequalities, we get 
\[\frac{1}{2}(n_i+n_{i+1})c_2k\geq \frac{c}{2}n_in^{\frac{1}{k}}\]
so 
\[\frac{n_{i+1}}{n_i}\geq \frac{c}{c_2k}n^{\frac{1}{k}}-1\geq \frac{c}{2c_2k}n^{\frac{1}{k}}\]
for $n$ sufficiently large. So the lemma is proved.
\end{proof}

Applying the lemma, we get that 
\[n_k\geq (\frac{c}{2c_2}n^{\frac{1}{k}})^k=(\frac{c}{2c_2})^kn\]
and   this is a contradiction.
\end{proof}

%\subsection{Remark}
\begin{rmk}
We see that this gives a better upper bound than that obtained by excluding copies of $K_{k,k}$, best known to be $O(n^{2-\frac{1}{k}})$.
\end{rmk}

\section{Avoiding general subgraphs}
We are now going to deal with the general case. We will show that the behaviour of the extremal function $ex(n, H)$ is tied intimately to the chromatic number of the graph $H$. To prove the Erdős-Stone-Simonovits theorem, we will first prove the following lemma in an original and elementary way. A more general variant of this lemma appears in \cite{nikiforov2008} and \cite{Nikiforov2007SpectralESB}, but our proof is original and, to the best of our knowledge, has not been published elsewhere.

\begin{lem}\label{lem-Lemma}

For any natural numbers $r$ and $t$ and any positive $\epsilon$ with $\epsilon < \frac{1}{r}$, there exists an $N$
such that the following holds: any graph $G$ with $n\geqslant N$ vertices and $(1-\frac{1}{r}+\epsilon)\frac{n^2}{2}$ edges contains $r + 1$ disjoint sets of vertices $W_1, . . . , W_{r+1}$ of size $t$ such that the graph between $W_i$ and $W_j$, for every $1\leqslant i < j\leqslant r + 1$, is complete.

\end{lem}

\begin{proof} [Proof of the Lemma~\ref{lem-Lemma}]

 To begin, we find a subgraph $H$ of $G$ such that every degree in $H$ is at least 
 \[(1-\frac{1}{r}+\frac{\epsilon}{2})|V(H)|.\]
To find such a graph, we remove one vertex at a proper time. If, in this process, we reach a graph with $m$ vertices and there is some vertex which has fewer than $(1-\frac{1}{r}+\frac{\epsilon}{2})m$ neighbors in this graph, we remove it.
Suppose that this process terminates when we have reached a graph $H$ with $p$ vertices. To show that $p$ is not too small, consider $x$ be the total number of edges that have been removed from the graph. When the graph has $m$ vertices, we remove at most $ (1-\frac{1}{r}+\frac{\epsilon}{2})m$ edges. Therefore, the total number of edges removed is: 
\[\sum_{i=p+1}^n\limits (1-\frac{1}{r}+\frac{\epsilon}{2})i=(1-\frac{1}{r}+\frac{\epsilon}{2})\frac{(p+n+1)(n-p)}{2}\leqslant (1-\frac{1}{r}+\frac{\epsilon}{2})\frac{n^2-p^2}{2}+\frac{n-p}{2}.\]

We know that $|E(H)|\leqslant \frac{p^2}{2}$ so we have 
\begin{align*}|E(G)| & \leqslant   \frac{p^2}{2}+(1-\frac{1}{r}+\frac{\epsilon}{2})\frac{n^2-p^2}{2}+\frac{n-p}{2} = \\
& = (1-\frac{1}{r}+\frac{\epsilon}{2})\frac{n^2}{2}+(\frac{1}{r}-\frac{\epsilon}{2})\frac{p^2}{2}+\frac{n-p}{2}.
\end{align*}

But we also have 
\[|E(G)|\geqslant (1-\frac{1}{r}+\epsilon)\frac{n^2}{2}.
\]

Therefore, the process stops once

\[(1-\frac{1}{r}+\epsilon)\frac{n^2}{2}> (1-\frac{1}{r}+\frac{\epsilon}{2})\frac{n^2}{2}+(\frac{1}{r}-\frac{\epsilon}{2})\frac{p^2}{2}+\frac{n-p}{2}\]
equivalent with 
\[\epsilon\frac{n^2}{4}-\frac{n}{2}> (\frac{1}{r}-\frac{\epsilon}{2})\frac{p^2}{2}-\frac{p}{2} \]
and this is clearly true for every $p\leqslant \frac{r\epsilon}{3}n$. From now on, we will assume that we are working within this large well behaved subgraph $H$ of graph $G$.

We will show, by induction on $r$, that there are $r+1$ sets $W_1, W_2, . . . , W_{r+1}$ of size $t$ such that every
edge between $W_i$ and $W_j$ , with $1\leqslant i < j \leqslant r + 1$, is in $H$. For $r = 0$, there is nothing to prove.
Given $r > 0$ and $s = \lceil \frac{3t}{\epsilon} \rceil$, we apply the induction hypothesis to find $r$ disjoint sets $L_1, L_2, . . . ,L_r$ of size $s$ such that the graph between every two disjoint sets is complete. Let 
\[U = V(H) \backslash \{L_1\cup \ldots \cup L_r\}\]
and let $R$ be the set of vertices in $U$ which are adjacent to at least $t$ vertices in each $L_i$.
We are going to estimate the number of edges missing between $U$ and $L_1\cup\ldots\cup L_r$. Since every vertex
in $U \backslash R$ is adjacent to fewer than $t$ vertices in some $L_i$, we have that the number of missing edges is
at least 
\[ |U\backslash R|(s-t)\geqslant (p-rs-|R|)(1-\frac{\epsilon}{3})s.\]

On the other hand, every vertex in $H$ has at most missing $ (\frac{1}{r}-\frac{\epsilon}{2})p$ edges. Therefore, counting over all vertices in $L_1\cup \ldots \cup  L_r$, we have at most  $ rs(\frac{1}{r}-\frac{\epsilon}{2})p$ missing edges  between $U$ and $L_1\cup\ldots\cup L_r$.
Therefoe
\[rs(\frac{1}{r}-\frac{\epsilon}{2})p\geqslant (p-rs-|R|)(1-\frac{\epsilon}{3})s\]
so we have that 
\[|R|(1-\frac{\epsilon}{3}) \geqslant (p-rs)(1-\frac{\epsilon}{3})-r(\frac{1}{r}-\frac{\epsilon}{2})p\] 
so 
\[|R|(1-\frac{\epsilon}{3})\geqslant \epsilon(\frac{r}{2}-\frac{1}{3})p-(1-\frac{\epsilon}{3}).\]

Since $\epsilon$, $r$ and $s$ are constants, we can make $|R|$ large by making $p$ large (by making $N$ larger). In particular, we may make $|R|$ such that $|R|>{\binom{s}{t}}^r(t-1)$.

Every element in $ R$ has at least $t$ neighbors in each $L_i$. There are at most ${\binom{s}{t}}^r$ ways to choose a
$t$-element subset from each of $L_1\cup \ldots\cup L_r$. By the pigeonhole principle and the size of $|R|$, there must be some subsets $W_1, \ldots , W_r$ and a set $W_{r+1}$ of size $t$ from $R$ such that every vertex in $W_{r+1}$
is connected to every vertex in $W_1\cup \ldots \cup W_r$. Since $W_1, \ldots , W_r$ are already joined in the appropriate
manner, this completes the proof.
\end{proof}

We can now quickly deduce the Erdős-Stone-Simonovits theorem.

\begin{thm}[Erdős-Stone-Simonovits]\label{thm-Erdős-Stone-Simonovits}

For any fixed graph $G$ and any fixed $\epsilon >0$, there is N
such that, for any $n\geqslant N$, we have 
\[ex(n,G)\leqslant \frac{n^2}{2}(1-\frac{1}{\chi(G)-1}+\epsilon).\]
\end{thm}

\smallskip

\begin{proof}[Proof of the Theorem]~\ref{thm-Erdős-Stone-Simonovits} Note that if $ G$ has chromatic number $\chi(G)$, then, provided $t$ is large enough,
it can be embedded in a graph $H$ consisting of $\chi(G)$ sets $W_1, W_2, \ldots, W_{\chi(G)}$ of size $t$ such that the
graph between any two disjoint $W_i$ and $W_j$ is complete. We simply embed any given color class into
a different $W_i$. The theorem now follows from an application of the previous Lemma.

\end{proof}

\nocite*

%\noindent {\it Stefan Gobej}\\  
%Vlaicu Voda High School\\
%Curtea de Arges, Romania \\
%E-mail: {\tt stefangobej@gmail.com}\\ \\ 

\end{document}